\documentclass[11pt]{amsart}
\usepackage{amsfonts,amssymb,amscd,amsmath,enumerate,verbatim,calc,graphicx}

\usepackage[numbers]{natbib}         

\def\frk{\mathfrak}               

\def\Phi{{\frk N}}
\def\opn#1#2{\def#1{\operatorname{#2}}} 
\opn\con{conv} \opn\relint{relint} \opn\vol{vol} \opn\heit{ht} 

\newtheorem*{Ack}{Acknowledgments}

\numberwithin{equation}{section}

\newtheorem{theorem}{Theorem}
\newtheorem{proposition}[theorem]{Proposition}
\newtheorem{lemma}[theorem]{Lemma}
\newtheorem{cor}[theorem]{Corollary}

\newtheorem{example}[theorem]{Example}
\newtheorem{remark}[theorem]{Remark}

\newtheorem{definition}[theorem]{Definition}

\numberwithin{theorem}{section}

\newcommand{\card}[1]{{\mid\! #1 \!\mid}}

\def\vol{{\rm vol}}
\def\lra{\Leftrightarrow}

\DeclareMathOperator{\codeg}{codeg}

\DeclareMathOperator{\Vol}{Vol}

\def\ra{{\Rightarrow}}

\def\Pb{{\overline{P}}}
\def\ehr{{\rm ehr}}
\def\iZ{{\intr_\Z}}

\newcommand\mcd{\mathrm{mcd}}
\newcommand\md{\mathrm{md}}

\newcommand\MV{\mathrm{MV}}
\newcommand\R{\mathbb{R}}

\newcommand\Z{\mathbb{Z}}
\newcommand\C{\mathbb{C}}
\newcommand\conv{\mathrm{conv}}
\newcommand\intr{\mathrm{int}}

\begin{document}

\title[The mixed degree of families of lattice polytopes]{The mixed degree of families of lattice polytopes}

\author{Benjamin Nill}
\address{
Benjamin Nill, Fakult\"at f\"ur Mathematik, 
Otto-von-Guericke-Universit\"at Magdeburg, 
Universit\"atsplatz 2, 
39106 Magdeburg, Germany
}

\email{benjamin.nill@ovgu.de}

\keywords{mixed volume, Ehrhart polynomials, lattice polytopes}

\begin{abstract}The degree of a lattice polytope is a notion in Ehrhart theory that was studied quite intensively over the previous years. It is well-known that a lattice polytope has normalized volume one if and only if its degree is zero. Recently, Esterov and Gusev gave a complete classification result of families of $n$ lattice polytopes in $\R^n$ whose mixed volume equals one. Here, we give a reformulation of their result involving the novel notion of a mixed degree that generalizes the degree similar to how the mixed volume generalizes the volume. We discuss and motivate this terminology, and explain why it extends a previous definition of Soprunov. We also remark how a recent combinatorial result due to Bihan solves a related problem posed by Soprunov.
\end{abstract}

\maketitle

~\vspace{-0.7cm}\section{Definitions and motivation}

\subsection{Introduction} Lattice polytopes in $\R^n$ are called {\em hollow} (or {\em lattice-free}) \cite{NillZiegler, Santos} if they have no lattice points (i.e., elements in $\Z^n$) in their relative interiors. In this paper, we initiate the study of large families of lattice polytopes with hollow Minkowski sums. We observe that such a family can consist of at most $n$ elements (Proposition~\ref{nonneg}). In Theorem~\ref{md0-low}, we deduce from the main result in \cite{EG12} that a family of $n$ lattice polytopes in $\R^n$ has mixed volume one if and only if the Minkowski sums of all subfamilies are hollow. In order to measure the `hollowness' of a family of lattice polytopes, we introduce the mixed degree of a family of lattice polytopes. Our goal is to convince the reader that this is a worthwhile to study invariant of a family of lattice polytopes that naturally generalizes the much-studied notion of the degree of a lattice polytope in a manner similar to how the mixed volume generalizes the normalized volume (see Subsection~\ref{sec:ehrhart}). As first positive evidence for this claim, we show the nonnegativity of the mixed degree (Subsection~\ref{sec:nonneg}), a generalization of the nonnegativity of the degree, and the characterization of mixed degree zero by mixed volume one (Subsection~\ref{sec:mixedzero}) in analogy to the characterization of degree zero by normalized volume one. We will also explain how the definition given here generalizes an independent definition of Soprunov (Subsection~\ref{soprunov-sec}).

\subsection{Basic definitions} Let us recall that a {\em lattice polytope} $P \subset \R^n$ is a polytope whose vertices are elements of the lattice $\Z^n$. Two lattice polytopes are {\em unimodularly equivalent} if they are isomorphic via an affine lattice-preserving transformation. We denote by $\conv(A)$ the convex hull of a set $A \subseteq \R^n$. We say $P$ is an $n$-dimensional {\em unimodular simplex} if it is unimodularly equivalent to 
$\Delta_n := \conv(0,e_1,\ldots,e_n)$, where $0$ denotes the origin of $\R^n$ and $e_1, \ldots, e_n$ the standard basis vectors. We define the {\em normalized volume} $\Vol(P)$ as 
$\dim(P)!$ times the Euclidean volume with respect to the affine lattice given by the intersection of $\Z^n$ and the affine span of $P$. Note that $\Vol(\Delta_n)=1$.
 
\begin{definition}{\rm 
Let $P_1, \ldots, P_m \subset \R^n$ be a finite set of lattice polytopes.
\begin{itemize}
\item For $k \in \Z_{\ge 1}$ we set $[k] := \{1, \ldots, k\}$.
\item For $\emptyset \not= I \subseteq [m]$ we define their {\em Minkowski sum}
\[P_I := \sum_{i \in I} P_i := \left\{\sum_{i\in I} x_i \;:\; x_i \in P_i \; \text{ for } i \in I\right\}.\]
We set $P_\emptyset := \{0\}$.
\item For $\emptyset \not= A \subset \R^n$ we define
\[A_\Z := A \cap \Z^n,\quad\quad \intr_\Z(A) := \intr(A) \cap \Z^n,\]
where the interior always denotes the relative interior (i.e., the interior with respect to the affine span of $A$). Recall that the interior of a point is considered to be the point itself, i.e., it is non-empty. 
\item For convenience, we say $P_1, \ldots, P_m$ is {\em proper} in $\R^n$ if 
\[\dim(P_1) \ge 1,\quad \ldots,\quad \dim(P_m) \ge 1, \quad \dim(P_{[m]})=n.\]
One of the reasons for excluding points in a proper family is that adding a point to a family just results in a lattice translation of their Minkowski sum.
\end{itemize}
}
\end{definition}

Throughout the paper, we {\em identify} two families of lattice polytopes if they agree up to a simultaneous unimodular transformation of $\Z^n$, permutations of the factors, and (lattice) translations of the factors.

\smallskip
	Let us state our main definition.

\begin{definition}{\rm Let $P_1, \ldots, P_m \subset \R^n$ be a finite set of lattice polytopes. 
\begin{itemize}
\item We define the {\em mixed codegree} of $P_1, \ldots, P_m$ as follows:

\begin{itemize}
\item If there exists $\emptyset \not= I \subseteq [m]$ such that $\intr_\Z(P_I) \not= \emptyset$, then $\mcd(P_1, \ldots, P_m)$ is defined as the minimal cardinality of such $I$;
\item otherwise, $\mcd(P_1, \ldots, P_m) := m+1$.
\end{itemize}
Let us note that 
\[1 \le \mcd(P_1, \ldots, P_m) \le m+1.\]
\item We define the {\em mixed degree} of $P_1, \ldots, P_m$ as
\[\md(P_1, \ldots, P_m) := \dim(P_{[m]}) + 1  - \mcd(P_1, \ldots, P_m).\]
Note that a family where one of the lattice polytopes is a point automatically has 
mixed degree equal to the dimension of $P_{[m]}$. We remark that for a proper family
\begin{equation}
n-m \le \md(P_1, \ldots, P_m) \le n.
\label{trivial}
\end{equation}
\end{itemize}
}\label{md-def}
\end{definition}

\begin{example}{\rm Consider the following family in $\R^2$: 
\[P_1 = \conv(0,e_1),\quad P_2 = \conv(0,e_2),\quad P_3 = \conv(0,e_1,e_2,e_1+e_2),\]
where $e_1, e_2$ is the standard basis of $\R^2$. Then the Minkowski sums of any two of these three lattice polytopes are hollow, while the Minkowski sum of all three is not. Hence, $\mcd(P_1,P_2,P_3)=3$ and $\md(P_1,P_2,P_3)=0$.
}
\label{cool-example}
\end{example}

\subsection{Relation to the Ehrhart-theoretic degree} Let us explain where the definition of the mixed (co-)degree comes from. 
Let $P \subset \R^n$ be an $n$-dimensional lattice polytope. The {\em codegree} of $P$ is defined as the smallest positive $k$ such that $\iZ(k P) \not=\emptyset$, and the {\em degree} of $P$ is given as $n+1-\codeg(P)$. Hence, for $P_1 := P, \ldots, P_n := P$, we see\footnote{We warn the reader that $\mcd(P)$ is in general different from $\codeg(P)$, as well as $\mcd(P_1, \ldots, P_m)$ is in general different from $\codeg(P_1*\cdots*P_m)$ (see Definition~\ref{def-cayley}).} that 
\[\mcd(P_1, \ldots, P_n)=\codeg(P), \quad \md(P_1, \ldots, P_n)=\deg(P).\]
This {\em unmixed} situation has been studied rather intensively (e.g., \cite{BN07,Nil07,HNP09}) leading to applications and relations to the adjunction theory of polarized toric varieties \cite{SandraFirst,SandraAdjunction,Araujo}, dual defective toric varieties \cite{DN10}, and almost-neighborly point configurations \cite{NillPadrol}. We hope to eventually generalize some of the achieved results to the mixed situation.

Note that the degree of a lattice polytope $P$ is originally defined as the degree of the {\em $h^*$-polynomial} $h^*_P$, the numerator polynomial of the rational generating function of the Ehrhart polynomial of $P$ (e.g., \cite{BN07}). In this case, the relation between degree and codegree follows from Ehrhart-Macdonald reciprocity (see \cite[Remark 1.2]{BN07}). We remark that a priori there are several possibilities how to define a generalization of the degree to families of lattice polytopes. Here, we generalize the geometric notion of the codegree instead of the more algebraic definition of the degree. It would be very interesting to find an analogous natural interpretation for the mixed degree. Originally motivated by tropical geometry \cite{ST10}, there is current research to investigate a mixed version of the $h^*$-polynomial \cite{MixedPaper,RamanNew}, however, its properties are yet to be fully understood. We caution the reader that the degree of the mixed $h^*$-polynomial as defined in \cite{MixedPaper} is in general not equal to the mixed degree discussed here. For instance, for the one-element family $P_1 := P$ with $m=1$ the mixed degree equals $n$ or $n-1$ depending on whether $P$ has interior lattice points or not, while on the other hand equation \cite[(10)]{MixedPaper} implies that the degree of the mixed $h^*$-polynomial is in this case always equal to $n$ if $n$ is odd.

\label{sec:ehrhart}
\subsection{Motivation from algebraic geometry}

Given a proper family of lattice polytopes, it is natural to consider the following situation. We say $P_1, \ldots, P_m$ is {\em irreducible} if $\intr_{\Z}(P_I) = \emptyset$ for any $\emptyset \not= I \subsetneq [m]$ and $\intr_{\Z}(P_{[m]}) \not= \emptyset$. The study of irreducible families of given mixed degree turns up in the Batyrev-Borisov construction of mirror-symmetric Calabi-Yau complete intersections \cite{Bat94, BB94, BN08}. For this, let us call a proper family of lattice polytopes $P_1, \ldots, P_m$ a {\em reflexive} family if their Minkowski sum is a reflexive polytope (up to translation), e.g., a so-called nef-partition \cite{BB94}. In this case, let us choose generic Laurent polynomials $f_1, \ldots, f_m$ with Newton polytopes $P_1, \ldots, P_m$. Then the complete intersection $V$ of the closures of the hypersurfaces $\{f_i=0\} \subset (\C^*)^n$ in the toric Gorenstein Fano variety associated to $P_{[m]}$ is a {\em Calabi-Yau variety} of dimension $n-m$ (see \cite{Bat94,BB94}). Let us assume that the reflexive family is irreducible. In this case, Corollary~3.5 in \cite{BB94} implies that the dimension of $V$ equals the mixed degree minus one; 
$V$ is non-empty if and only if the mixed degree is at least one; and $V$ is an irreducible variety if and only if the mixed degree is at least two. By the so-called semi-simplicity principle for nef-partitions (Section~5 in \cite{BB94} and more generally Proposition~6.13 in \cite{BN08}), any reflexive family can be partitioned into irreducible reflexive subfamilies. Hence, in this toric setting the study of Calabi-Yau complete intersections of given dimension is closely related to the study of irreducible families of given mixed degree.

\subsection{Structure of the paper} Section~2 contains the main results of this paper. Proofs are give in Section~3.

\begin{Ack}{\rm The author would like to thank Alan Stapledon, Fr\'ed\'eric Bihan, Alexander Esterov, Alicia Dickenstein, Christian Haase, Raman Sanyal, and especially Ivan Soprunov for useful discussions and helpful remarks. The author is grateful to the MSRI and the Fields Institute for financial support. First ideas about a mixed degree go back to the stay of the author as a postdoctoral fellow at the MSRI as part of the program on Tropical Geometry. This paper was essentially finished during a stay at the Fields Institute as part of the program on Combinatorial Algebraic Geometry. The author is partially supported by the Vetenskapsr\aa det grant NT:2014-3991 (as an affiliated researcher with Stockholm University).
}
\end{Ack}

\section{Results on the mixed degree}

In this section we describe our results on the mixed degree of a family of lattice polytopes. We will postpone all proofs to the next section.

\subsection{Nonnegativity}
\label{sec:nonneg}

Here is our first observation.

\begin{proposition}
The mixed degree is nonnegative.
\label{nonneg}
\end{proposition}

Let us recall how one can convex-geometrically prove nonnegativity in the unmixed situation. Note that for an arbitrary interior point of an $n$-dimensional lattice polytope $P$ Carath\'eodory's theorem allows to find vertices $v_0, \ldots, v_n$ of $P$ such that the point is in the convex hull of these vertices. Therefore, also $(v_0 + \cdots + v_n)/(n+1)$ is in the interior of $P$. Hence, $\iZ((n+1)P) \not=\emptyset$, thus, $\codeg (P) \le n+1$, so $\deg(P) \ge 0$. Hence, Proposition~\ref{nonneg} may be seen as a mixed version of Carath\'eodory's theorem in the following sense: {\em Given  $P_1, \ldots, P_m$ lattice polytopes in $\R^n$ with $m > n$, there exists a non-empty subset $I \subseteq [m]$ of cardinality $|I| \le n+1$ such that the Minkowski sum $P_I$ contains a relative interior lattice point.}

\subsection{Mixed degree zero} In the unmixed case, $\deg(P)=0$ if and only if $\Vol(P)=1$ (see \cite{BN07}). As we will see, the analogous statement is also true in the mixed situation. This may be regarded as favorable evidence that the definition of the mixed degree is a reasonable generalization of the unmixed degree of a lattice polytope.

\label{sec:mixedzero}

For this, let us define the (normalized) {\em mixed volume} $\MV(P_1, \ldots, P_n)$ of a family $P_1, \ldots, P_n \subset \R^n$ 
as the coefficient of $\lambda_1 \cdots \lambda_n$ of the homogeneous polynomial $\vol_n(\lambda_1 P_1 + \cdots + \lambda_n P_n)$, where $\vol_n$ is the standard Euclidean volume of $\R^n$, see \cite{EG12,Sch93}. It is nonnegative, monotone with respect to inclusion, and multilinear. Note that the mixed volume defined here is normalized such that for an $n$-dimensional lattice polytope $P \subset \R^n$ we have $\MV(P, \ldots, P) = \Vol(P)$. Hence, the following result generalizes the unmixed statement. 

\begin{theorem}
Let $P_1, \ldots, P_n$ be a proper family of lattice polytopes in $\R^n$. Then the following conditions are equivalent:
\begin{enumerate}
\item $\md(P_1, \ldots, P_n)=0$
\item $\MV(P_1, \ldots, P_n)= 1$.
\end{enumerate}
\label{md0-low}
\end{theorem}

While the implication (1) $\Rightarrow$ (2) has a short proof, the reverse implication (2) $\Rightarrow$ (1) relies on the highly non-trivial classification of $n$ lattice polytopes of mixed volume one by Esterov and Gusev \cite{EG12}. It would be desirable to find a direct, classification-free proof. 

\medskip

In the unmixed case, there is only one $n$-dimensional lattice polytope of degree $0$, respectively normalized volume $1$, namely, the unimodular $n$-simplex. Such a uniqueness result also holds in the mixed case if all lattice polytopes in the family are full-dimensional. This was essentially first proven in \cite[Prop.~2.7]{CCDDRS11}. 

\begin{proposition}[Cattani et al. '13] Let $P_1, \ldots, P_m$ be $n$-dimensional lattice polytopes. Then $\md(P_1, \ldots, P_m) = 0$ if and only if $m \ge n$ and $P_1, \ldots, P_m$ equal the same unimodular $n$-simplex (up 
to translations).
\label{md0}
\end{proposition}

In the low-dimensional case, the situation is more complicated. From the results of Esterov and Gustev \cite{EG12} we get an inductive description of families of $n$ lattice polytopes of mixed degree zero. For this, let us define for $\emptyset \not= I \subseteq [m]$ the {\em lattice projection $\pi_I$} along the affine span of $P_I$. More precisely, $\pi_I$ is the $\R$-linear map induced by the lattice surjection $\Z^n \to \Z^n/\Gamma$, where $\Gamma$ is the subgroup that is a translate of the set of lattice points in the affine hull of $P_I$.

\begin{cor}
Let $P_1, \ldots, P_n$ be proper. Then $\md(P_1, \ldots, P_n)=0$ if and only if one of the following two cases holds:
\begin{enumerate}
\item $P_1, \ldots, P_n$ are contained in the same unimodular $n$-simplex (up to translations),
\item there exists an integer $1 \le k  < n$ such that (up to translations and permutation of $P_1, \ldots, P_n$) $P_1, \ldots, P_k$ are contained in a $k$-dimensional subspace of $\R^n$ with 
$\dim(P_{[k]})=k$ such that $\md(P_1, \ldots, P_k) = 0$ and $\md(\pi_{[k]}(P_{k+1}), \ldots, \pi_{[k]}(P_n))=0$.
\end{enumerate}
\label{inductive-mdeg0}
\end{cor}

For $m>n$ we do not yet have such a complete classification result of all families of $m$ lattice polytopes of mixed degree $0$. However, we can show that there are essentially only finitely many cases. 

\begin{theorem}
Let $P_1, \ldots, P_m$ be a proper family with $\md(P_1, \ldots, P_m)=0$ and $m > n$. Then one of the following two cases holds:
\begin{enumerate}
\item either $P_1, \ldots, P_m$ are contained in a unimodular $n$-simplex $Q$ (up to translations),
\item or the family $P_1, \ldots, P_m$ belongs to a finite number of exceptions (whose number depends only on $n$).
\end{enumerate}
Moreover, in the first case, at most $(2^n-1)(n-1)$ of the polytopes in the family are not equal to $Q$ (up to translations). More precisely, no face of $Q$ of dimension $j<n$ appears among $P_1, \ldots, P_m$ more than $j$ times (up to translations).
\label{mdeg0-guess}
\end{theorem}

\begin{remark}{\rm We leave it as an exercise to the reader to show that for $n=2$ there is precisely one exception in Theorem~\ref{mdeg0-guess}, namely, the family given in Example~\ref{cool-example}. 
It would be interesting to know whether there are exceptional families in Theorem~\ref{mdeg0-guess} of length larger than $n+1$.
}
\end{remark}

\subsection{Mixed degree at most one}
\label{soprunov-sec}

The following lower bound theorem can be found in \cite{B++08} based upon \cite{Sop07}.

\begin{theorem}[Soprunov '07] Let $P_1, \ldots, P_n$ be $n$-dimensional lattice polytopes. Then
\[|\intr_\Z(P_{[n]})| \ge \MV(P_1, \ldots, P_n) - 1.\]
\label{lower-bound}
\end{theorem}

The original proof of Theorem~\ref{lower-bound} involved the Euler-Jacobi Theorem and Bernstein's Theorem. In \cite[Problem 1]{B++08} Soprunov asked whether there is a purely combinatorial proof. We can affirmatively answer this question in Section~\ref{sopru-sec} by reducing it to a recent result by Bihan related to the nonnegativity of the so-called discrete mixed volume \cite{Bihan}.

\begin{example}{\rm Note that the full-dimensionality assumption in Soprunov's lower bound theorem cannot be removed. 
Consider in $\R^2$ a unimodular $2$-simplex $P_1$ and a line segment $P_2$ parallel to one of the edges of $P_1$. If $P_2$ contains $k$ lattice points, 
then $\MV(P_1,P_2)=k-1$, while $|\intr_\Z(P_1 + P_2)| = 0$.}
\label{two-dim-example}
\end{example}

In the unmixed case ($P_1 = \cdots = P_n = P$), Theorem~\ref{lower-bound} follows directly from Ehrhart theory, see \cite{B++08}. Moreover, Soprunov observes in his note that equality is attained if and only if $\deg(P) \le 1$. This observation led him to define in \cite{B++08} a family of $n$-dimensional lattice polytopes $P_1, \ldots, P_n$ as having mixed degree at most $1$ if equality in Theorem~\ref{lower-bound} is attained, and mixed degree $0$ if $P_{[n]}$ has no interior lattice points. As the following result shows, this is compatible with our definition.

\begin{proposition}
Let $P_1, \ldots, P_n$ be $n$-dimensional lattice polytopes. 

Then $|\intr_\Z(P_{[n]})| = \MV(P_1, \ldots, P_n) - 1$ if and only if $\md(P_1, \ldots, P_n) \le 1$.
\label{equality}
\end{proposition}

\section{Proofs}
\label{proofs-sec}

\subsection{Nonnegativity}

This will be a simple consequence of basic properties of the mixed volume. For this, let us recall a well-known alternative formula (see e.g. \cite{KS08}).

\begin{proposition}
Let $P_1, \ldots, P_n$ be lattice polytopes in $\R^n$. Then
\[\MV(P_1, \ldots, P_n) = \sum_{I \subseteq [n]} (-1)^{n-\card{I}} \;|P_I \cap \Z^n|.\]
\label{mv-all}
\end{proposition}

Recall that $|P_\emptyset \cap \Z^n|=1$. 
Using reciprocity one gets another slightly less well-known formula involving interior lattice points.

\begin{cor}
Let $P_1, \ldots, P_n$ be lattice polytopes in $\R^n$. Then
\[\MV(P_1, \ldots, P_n) = 1+ \sum_{\emptyset \not= I \subseteq [n]} (-1)^{\dim(P_I)-\card{I}} \; |\intr_\Z(P_I)|.\]
In particular, if $P_1, \ldots, P_n$ are $n$-dimensional, then
\[\MV(P_1, \ldots, P_n) = 1 + \sum_{\emptyset \not= I \subseteq [n]} (-1)^{n-\card{I}} \;|\intr_\Z(P_I)|.\]
\label{mv-innere}
\end{cor}

\begin{proof}
Let us denote by $\ehr_P(t)$ the Ehrhart polynomial of a lattice polytope $P \subset \R^n$, i.e., 
$\ehr_P(t) = |(t P) \cap \Z^n|$ for $t \in \Z_{\ge 1}$. Ehrhart-Macdonald reciprocity yields 
\[\ehr_P(-1)=(-1)^{\dim(P)} |\iZ(P)|.\]
(For the case of dimension $0$, recall that the interior of a lattice point is the lattice point itself.) 
Applying Proposition~\ref{mv-all} to $t P_1, \ldots, t P_n$ for $t \in \Z_{\ge 1}$ gives
\[t^n \MV(P_1, \ldots, P_n) = \MV( t P_1, \ldots, t P_n) = (-1)^n + \sum_{\emptyset \not= I \subseteq [n]} (-1)^{n-\card{I}} \; \ehr_{P_I}(t).\]
Plugging in $t=-1$ and Ehrhart-Macdonald reciprocity yields
\[(-1)^n \MV(P_1, \ldots, P_n) = (-1)^n + \sum_{\emptyset\not= I \subseteq [n]} (-1)^{n-\card{I}} \; (-1)^{\dim(P_I)} |\intr_\Z(P_I)|.\]
\end{proof}

\begin{proof}[Proof of Proposition~\ref{nonneg}]

It follows from the definition of the mixed degree that it suffices to show nonnegativity for a proper family $P_1, \ldots, P_{n+1}$ in $\R^n$. We assume that $\iZ(P_I)=\emptyset$ for any $\emptyset \not=I \subseteq [n+1]$. Consider the proper family $P_1, \ldots, P_{n-1},P_n+P_{n+1}$. Corollary~\ref{mv-innere} and multilinearity of the mixed volume yield 
\[1 = \MV(P_1, \ldots, P_{n-1},P_n+P_{n+1}) = \MV(P_1, \ldots, P_{n-1},P_n) + \MV(P_1, \ldots, P_{n-1},P_{n+1}).\]
However, Corollary~\ref{mv-innere} also implies that both of these summands equal $1$, a contradiction.
\end{proof}

\subsection{Mixed degree $0$ -- the full-dimensional case}

Since it might be of independent interest, we provide several characterizations of this situation. For this, let us recall the following definition.

\begin{definition}{\rm
The {\em Cayley polytope} of lattice polytopes $P_1, \ldots, P_m$ in $\R^n$ is defined as
\[P_1 * \cdots * P_m := \conv(P_1 \times \{e_1\}, \ldots, P_m \times \{e_m\}) \subset \R^{n+m}\]
where $e_1, \ldots, e_m$ is the standard basis of of $\R^m$. Note that if $\dim(P_{[m]})=n$, then \[\dim(P_1 * \cdots * P_m)=n+m-1.\]
\label{def-cayley}}
\end{definition}

\begin{proposition} Let $P_1, \ldots, P_n$ be $n$-dimensional lattice polytopes in $\R^n$. Then the following conditions are equivalent:
\begin{enumerate}
\item $\md(P_1, \ldots, P_n) =0$
\item $\iZ(P_1 + \cdots + P_n) = \emptyset$
\item $\MV(P_1, \ldots, P_n)=1$
\item $P_1, \ldots, P_n$ are lattice translates of the same unimodular $n$-simplex
\item $\Vol(P_1 + \cdots + P_n) = n^n$, which is the minimal possible value
\item $\Vol(P_1 * \cdots * P_n) = \binom{2n-1}{n}$, which is the minimal possible value
\item $\deg(P_1 * \cdots * P_n) = n-1$, which is the minimal possible value
\end{enumerate}
We remark that otherwise $\deg(P_1 * \cdots * P_n) =n$.
\label{md0-char}
\end{proposition}

\begin{proof}

(1) $\lra$ (2) follows from full-dimensionality. (1) $\Rightarrow$ (3) by Corollary~\ref{mv-innere}. (3) $\Rightarrow$ (4) was proven in \cite[Prop.~2.7]{CCDDRS11}. Clearly, (4) $\Rightarrow$ (1).

(5) $\ra$ (3) follows from an expression in terms of multinomial coefficients (e.g., \cite{Sch93}):
\[\Vol(P_1+\cdots+P_n) = \sum_{k_1 + \cdots + k_n=n} \binom{n}{k_1, \cdots, k_n} \cdot \MV(P_1^{(k_1)}, \ldots, P_n^{(k_n)}),\]
where the sum is over all nonnegative integer $n$-tuples $k_1, \ldots, k_n$ satisfying the condition $k_1 + \cdots + k_n=n$; moreover, $P_i^{(k_i)}$ means that $P_i$ should be repeated $k_i$ times. 
Since all lattice polytopes are full-dimensional, each of the mixed volumes in the sum is positive. Let us note that, if they are all equal to $1$, the right side equals $n^n$. For (3) $\ra$ (5) note that if $\MV(P_1, \ldots, P_n)=1$, then also each of the mixed volumes (since they are all positive) must be equal to $1$ by the Alexandrov-Fenchel inequality (e.g., \cite{Sch93}). Alternatively, one can directly verify (4) $\ra$ (5). 

(6) $\ra$ (4) uses the following formula (e.g., \cite{DK86}): $\Vol(P_1 * \cdots * P_n)$ equals the sum of $\MV(P_{i_1}, \ldots, P_{i_n})$ over all possible choices of unordered $n$-tuples $i_1, \ldots, i_n \in [n]$, where repetitions are allowed (there are $\binom{2n-1}{n}$ such choices). Since all lattice polytopes are full-dimensional, each of the mixed volumes in the sum is positive. The converse (3) $\ra$ (6) follows as above from the Alexandrov-Fenchel inequality or directly by checking $(4) \ra (6)$.

(2) $\lra$ (7) is a consequence of the so-called Cayley-Trick. Consider the lattice projection $\pi$ mapping $\,P_1 *  \cdots * P_n$ onto $\Delta_{n-1}$. Therefore, $\codeg(P_1 * \ldots * P_n) \ge 
\codeg(\Delta_{n-1})=n$. Now, the intersection of $n (P_1 *  \cdots * P_n)$ 
with the preimage of the unique interior lattice point in $n \Delta_{n-1}$ is unimodularly equivalent to $P_1 + \cdots + P_n$.  Therefore, we have $\codeg(P_1 *  \cdots * P_n) > n$ (or equivalently, $\deg(P_1 * \cdots * P_n) < n$) precisely when (2) is satisfied. Let us note that in this case, by (4), $P_1 *  \cdots * P_n \cong \Delta_n \times \Delta_{n-1}$, so its degree equals $n-1$.
\end{proof}

\begin{proof}[Proof of Proposition~\ref{md0}]

The implication follows from \eqref{trivial} and by applying Proposition~\ref{md0-char} to any subfamily of $n$ lattice polytopes. The reverse implication is a direct consequence of $\codeg(\Delta_n)=n+1$. 

\end{proof}

\subsection{Mixed degree $0$ -- the low-dimensional case}

We used before that the mixed volume of full-dimensional polytopes is positive. Bernstein's criterion (e.g., \cite{Sch93}) gives the precise generalization.

\begin{lemma}
$\MV(P_1, \ldots, P_n) \ge 1$ if and only if $\dim(P_I) \ge |I|$ for all $\emptyset \not= I \subseteq [n]$.
\label{mv0-lemma}
\end{lemma}

Note that in this case $P_1, \ldots, P_n$ is necessarily a proper family. Let us also recall the following well-known fact about mixed volumes (e.g., \cite{EG12,Sch93}).

\begin{lemma} Let $P_1, \ldots, P_n$ be lattice polytopes in $\R^n$. If $P_1, \ldots, P_k$ (for $1 \le k\le n$) are contained in a $k$-dimensional rational subspace $L$ of $\R^n$, then 
\[\MV(P_1, \ldots, P_n) = \MV(P_1, \ldots, P_k) \cdot \MV(\Pb_{k+1}, \ldots, \Pb_n),\]
where $\Pb_i$ is the image of $P_i$ under the projection along $L$.
\label{mv-induction}
\end{lemma}

Here is the (inductive) characterization of families with mixed volume one.

\begin{theorem}[Esterov, Gusev '12]
Let $P_1, \ldots, P_n$ be lattice polytopes in $\R^n$. 

Then $\MV(P_1, \ldots, P_n)=1$ if and only if $\MV(P_1, \ldots, P_n)\not=0$, and there exists an integer $1 \le k \le n$ such that, up to translations, 
$k$ of the polytopes are faces of the same unimodular $k$-simplex $Q$, and the projection of the other $n-k$ simplices along $Q$ form a family of mixed volume one.
\label{mv=1}
\end{theorem}

Note that the `if'-direction follows from Lemmas~\ref{mv0-lemma} and \ref{mv-induction} together with the monotonicity of the mixed volume, while the `only if'-direction is a highly non-trivial result.

\smallskip

For the proof of Theorem~\ref{md0-low} we need the following simple observation.

\begin{lemma}
Let $P_1, \ldots, P_k$ be faces of the unimodular simplex $\Delta_n := \conv(0,e_1, \ldots, e_n)$ such that 
$\dim(P_I) \ge |I|$ for any $\emptyset \not= I \subseteq [k]$. Then $\iZ(P_I) = \emptyset$ for any $\emptyset \not= I \subseteq [k]$.
\label{key}
\end{lemma}

\begin{proof}
Let $\emptyset \not= I \subseteq [k]$, and $j := |I|$. We note that $P_I \subseteq j \Delta_n$. Hence, there exists a unique face $F$ of $j \Delta_n$ such that $\intr(P_I) \subseteq \intr(F)$. 
Let $d := \dim(F)$, so $F \cong j \Delta_d$. Since by assumption $1 \le j \le \dim(P_I) \le d$, we get $\iZ(j \Delta_d) = \emptyset$, hence, $\iZ(P_I)= \emptyset$.
\end{proof}

\begin{proof}[Proof of Theorem~\ref{md0-low}]

The direction (1) $\ra$ (2) follows directly from Corollary~\ref{mv-innere}. For (2) $\ra$ (1) 
we can assume by Theorem~\ref{mv=1} and Lemma~\ref{mv0-lemma} that $P_1, \ldots, P_k$ (for some $1 \le k \le n$) are faces 
of the unimodular simplex $\Delta_k \subset \R^k$, 
$\dim(P_{[k]}) = k$, and $\dim(P_I) \ge |I|$ for any $\emptyset \not= I \subseteq [k]$. In particular, Lemma~\ref{key} yields $\iZ(P_I) = \emptyset$ for any $\emptyset \not= I \subseteq [k]$. 
This proves the statement for $k=n$, so let 
$k<n$. Considering the projection along $\R^k$, Theorem~\ref{mv=1} implies $\MV(\Pb_{k+1}, \ldots, \Pb_n)=1$, thus, $\md(\Pb_{k+1}, \ldots, \Pb_n) =0$ by induction. Let us assume that there exists $ \emptyset \not=I \subseteq [n]$ such that $\iZ(P_I) \not=\emptyset$, in particular, $I \not\subseteq [k]$. We observe that $\Pb_I$ equals $\Pb_{I \cap \{k+1, \ldots, n\}}$ up to a translation. Hence, $\iZ(\Pb_{I \cap \{k+1, \ldots, n\}}) \not=\emptyset$, a contradiction to $\md(\Pb_{k+1}, \ldots, \Pb_n) = 0$.
\end{proof}

\begin{remark}{\rm Let us note the following observation: If $P_1, \ldots, P_k$ in $\R^n$ such that $\mcd(P_1, \ldots, P_k) = k+1$, then $k \le \dim(P_{[k]})$ by nonnegativity of the mixed degree.}
\label{dim-remark}
\end{remark}

\begin{proof}[Proof of Theorem~\ref{mdeg0-guess}]

We first consider the case $m=n+1$. Here, $\iZ(P_I)=\emptyset$ for any $\emptyset \not=I \subsetneq [n+1]$, and $\iZ(P_{[n+1]}) \not=\emptyset$. In fact, since as in the proof of Proposition~\ref{nonneg} we have $\MV(P_1, \ldots, P_{n-1},P_n+P_{n+1}) = 2$, Corollary~\ref{mv-innere} implies that $|\iZ(P_{[n+1]})|~=~1$. By a well-known result in the geometry of numbers \cite{LZ91} there are up to unimodular equivalence only a finite number of lattice polytopes with one interior lattice point in fixed dimension $n$. This implies that there are only finitely many families $P_1, \ldots, P_{n+1}$ with $\md(P_1, \ldots, P_{n+1})=0$ up to our identification. 

So, let $m>n+1$. Remark~\ref{dim-remark} applied to $P_1, \ldots, P_n$ implies $\dim(P_{[n]})=n$, so $P_1, \ldots, P_{n+1}$ is proper. Let us fix $P_1, \ldots, P_{n+1}$ as one of the finitely many types in above argument. Let $n+1<i\le m$. By similarly considering $P_1, \ldots, P_n, P_i$ we deduce that there are only finitely many possibilities (say, $N$ many) for $P_i$ up to translation. 
Note that $N$ only depends on $n$. 

Hence, we may assume that $m > n+1+(n-1) N$. By the pidgeonhole principle, there exist $P_{i_1}, \ldots, P_{i_n}$ (with $n+2 \le i_1 < \cdots < i_n \le n+2+(n-1) N$) that are all equal to the same lattice polytope $Q$ up to translations. Again, Remark~\ref{dim-remark} applied to $P_{i_1}, \ldots, P_{i_n}$ yields that $\dim(Q)=n$. Moreover, Corollary~\ref{mv-innere} implies that  $\Vol(Q)=\MV(Q, \ldots, Q){=1}$, i.e., $Q$ is a unimodular $n$-simplex.

Let $i \in [m]\backslash\{i_1, \ldots, i_n\}$ such that $P_i$ is not contained in $Q$ up to translations. We will show that this case cannot occur. Again, Lemma~\ref{mv-innere} yields  $\MV(Q, \ldots, Q, P_i)=1$ (where $Q$ is chosen $n-1$ times). Now, Theorem~\ref{mv=1} implies that there exists $1 \le k \le n$ such that $k$ of the polytopes $Q, \ldots, Q, P_i$ are contained up to translations in a $k$-dimensional unimodular simplex $S$ and the projection of the other $(n-k)$ polytopes along this simplex yields again a family of mixed volume one. 
Assume $k > 1$. In this case, one of the $Q$'s would be contained in $S$ up to translation, hence $S$ would be equal to $Q$ up to translation, so $k=n$, and $P_i$ would be contained in $Q$ up to translation, a contradiction. Therefore, $k=1$, and $P_i$ is contained in $S$ up to translation. Since $P_i$ is not a point, 
we see that $P_i=S$ must be a lattice interval containing two lattice points. Since projecting $Q, \ldots, Q$ along $P_i$ (via a lattice projection $\pi_i$) yields again a family of full-dimensional lattice polytopes of mixed volume one, Proposition~\ref{md0} implies that $\pi_i(Q)$ is an $(n-1)$-dimensional unimodular simplex. In particular, we see that there must be two vertices of $Q$ that get mapped to the same 
vertex of $\pi_i(Q)$. Hence, since $P_i$ lies in a fiber of $\pi_i$, we deduce that $P_i$ is up to a translation an edge of $Q$, again a contradiction.

Finally, let us consider the situation that all lattice polytopes are contained in a unimodular $n$-simplex $Q$ up to translations. Because of $\codeg(\Delta_j)=j+1$, no face of $Q$ of dimension $j<n$ can appear $j+1$ times. This proves the last statement in the theorem. It remains to observe the following easily verified binomial identity
\[\sum_{i=1}^{n-1} i \, \binom{n+1}{i+1} = (2^n-1)(n-1).\]
\end{proof}

\subsection{Mixed degree at most one}

\label{sopru-sec}

Let $P_1, \ldots, P_m$ be lattice polytopes in $\R^n$. Let us define for $\emptyset \not= I \subseteq [m]$  
\[g(I) := \sum_{\emptyset \not= J \subseteq I} (-1)^{\card{I}-\card{J}} \;|\intr_\Z(P_J)|.\]

\begin{theorem}[Khovanskii' 78, Bihan '14] 
If $P_1, \ldots, P_m$ are $n$-dimensional lattice polytopes, then $g([m])$ is nonnegative.
\label{genus-formula}
\end{theorem}

Bihan's proof \cite[Theorem~4.15(4)]{Bihan} is purely combinatorial.

\begin{remark}{\rm
Let us assume $m \le n$, and explain why nonnegativity follows from the algebro-geometric meaning of $g([m])$. Given $P_1, \ldots, P_m$, these lattice polytopes are the Newton polytopes of generic Laurent polynomials $f_1, \ldots, f_m \in \C[x_1^{\pm}, \ldots, x_n^{\pm}]$. We consider the set $X$ of their common solutions in the algebraic torus $(\C^*)^n = (\C\backslash\{0\})^n$. Let $\bar{X}$ be its Zariski closure in the projective toric variety associated to the normal fan of $P_{[m]}$. Then $g([m])$ equals the geometric genus of $\bar{X}$ (i.e., $h^{n-m,0}(\bar{X})$), see \cite{Kho78}. 
}
\end{remark}

\begin{example}{\rm As we see again from Example~\ref{two-dim-example}, the full-dimensionality assumption cannot be removed from Theorem~\ref{genus-formula}. In this situation, $|\intr_\Z(P_{[2]})|=|\intr_\Z(P_1)| = 0$, while 
$|\intr_\Z(P_2)|$ can be arbitrarily large. Hence, $g([2])$ can be arbitrarily negative.}
\end{example}

For $\emptyset \not= I \subseteq [m]$, let us now consider the following variant of $g(I)$:

\[\tilde{g}(I) := \sum_{\emptyset \not= J \subsetneq I} (-1)^{\card{I}-1-\card{J}} \;|\intr_\Z(P_J)|.\]

\begin{lemma}
For $\emptyset \not= I \subseteq [m]$,
\[\tilde{g}(I)=\sum_{\emptyset \not= J \subsetneq I} g(J).\]
\label{f-lemma}
\end{lemma}

\begin{proof}
M\"obius-inversion states that
\[|\intr_\Z(P_I)|  = \sum_{\emptyset \not=J\subseteq I} g(J).\]
Therefore, the statement follows from $\tilde{g}(I) = |\intr_\Z(P_I)|-g(I)$.
\end{proof}

We can now give the combinatorial proof of Soprunov's lower bound theorem and the characterization of its equality case.

\begin{proof}[Proof of Theorem~\ref{lower-bound} and Proposition~\ref{equality}]

Let $P_1, \ldots, P_n$ be $n$-dimensional lattice polytopes. In this case, Theorem~\ref{genus-formula} implies $g(I) \ge 0$ for any $\emptyset \not= I \subseteq [n]$. Hence, Lemma~\ref{f-lemma} yields $\tilde{g}([n]) \ge 0$. Now, rewriting Corollary~\ref{mv-innere} yields
\[\MV(P_1, \ldots, P_n)-1 = g([n])=|\intr_\Z(P_{[n]})|-\tilde{g}([n]) \le |\intr_\Z(P_{[n]})|.\]
In particular, we have equality if and only if $\tilde{g}([n])=0$. By Lemma~\ref{f-lemma} this is equivalent to $g(I)=0$ for all $\emptyset \not= I \subsetneq [n]$. By the definition of $g(I)$, this just means that 
$|\intr_\Z(P_I)|=0$ for any $\emptyset \not= I \subsetneq [n]$ which is equivalent to mixed degree $\le 1$.
\end{proof}

\bibliographystyle{acm}
\bibliography{bibliography}

\end{document}